\documentclass[12pt,dvipsnames]{amsart}

\usepackage{amsfonts, amssymb, amsmath, amsthm}                               % AMS symbols
\usepackage{mathrsfs, mathtools}                                              % Math symbols
\usepackage{geometry}                                                         % Spacing
\usepackage{setspace}
\usepackage{xcolor}
\usepackage{graphicx}
\usepackage{hyperref}
\usepackage{slashed}
\usepackage{upgreek}
\usepackage[pagewise]{lineno}

\newtheorem{theorem}{Theorem} 	      	      	                              % Theorem environment
\newtheorem{lemma}[theorem]{Lemma}     	       	      	      	      	      % Lemma environment
\newtheorem{proposition}[theorem]{Proposition} 	      	      	      	      % Proposition environment
\newtheorem{definition}[theorem]{Definition} 	      	      	                % Definition environment
\newtheorem{remark}[theorem]{Remark}                                          % Remark environment
\newtheorem{problem}[theorem]{Problem}

\numberwithin{equation}{section}                                              % Equation numbering
\numberwithin{theorem}{section}                                               % Theorem numbering

\newcommand{\mc}[1]{\mathcal{#1}}                                             % Mathcal
\newcommand{\R}{\mathbb{R}}                                                   % Real numbers
\newcommand{\N}{\mathbb{N}}                                                   % Natural numbers

\newcommand\numberthis{\addtocounter{equation}{1}\tag{\theequation}}
\newcommand{\nb}{\nabla}
\newcommand{\nasla}{\slashed{\nabla}}

\allowdisplaybreaks
\begin{document}

\title{Interior control of waves on time dependent domains}

\author{Vaibhav Kumar Jena}
\address{School of Mathematical Sciences\\
Queen Mary University of London\\
London E1 4NS\\ United Kingdom}
\email{v.k.jena@qmul.ac.uk}

\begin{abstract}
We obtain a novel interior control result for wave equations on time dependent domains. This is done by deriving a suitable Carleman estimate and proving the corresponding observability inequality. We consider the wave equation with time dependent lower order coefficients, without any time analyticity assumptions. Moreover, we obtain improved control regions when compared with standard Carleman methods.
\end{abstract}

\maketitle

\section{Introduction}

In this article, we address an interior control problem for $(n+1)$-dimensional wave equations with time dependent lower order terms, on time dependent domains. Such domains are also known as domains with a moving boundary.
The main tool to solve this control problem is a suitable Carleman estimate. This estimate is a direct consequence of the geometric nature of the corresponding Carleman estimate from \emph{Shao} \cite{Arick}, which solves a boundary control problem for similar general waves on time dependent domains. We also adapt some techniques from the author's article \cite{jena}, which obtained improved interior control results for waves on static domains. We achieve the best known Carleman based interior control result for general waves on such time dependent domains.

First, we present the notion of a moving boundary domain as follows\footnote{See Definition \ref{def_U_domain_MB} for the precise definition.} 
\begin{equation*} \label{eq_def_U_MB}
\mc{U} := \bigcup_{ \tau \in \R } \Big( \{ \tau \} \times \Omega_\tau \Big) \text{,}
\end{equation*}
where each \( \Omega_\tau \subset \R^n \) is an open and bounded domain, and they vary smoothly along \( \tau \). 
On this domain we want to control the following PDE
\begin{equation} \label{eq_ctrl_sys_MB}
\begin{cases}
- \partial_{tt}^2 y + \Delta_x y - \nabla_\mc{X} y + q y = F \textbf{1}_W, \qquad & \text{on } \mc{U},\\
(y, \partial_t y) = (y_0^-,y_1^-), \qquad & \text{at initial time}, \\
y=0, \qquad & \text{on } \partial \mc{U},
\end{cases}
\end{equation}
with $(y_0^-,y_1^-) \in L^2(\cdot) \times H^{-1}(\cdot)$, and where $\mc{X} \in C^\infty(\bar{\mc{U}};\R^{1+n})$ is a vector field, $q \in C^\infty(\bar{\mc{U}})$ is the potential and $W \subset \mc{U}$ is an interior subset, $\textbf{1}_W$ denotes the characteristic function over $W$, and $F$ is a forcing term. The control problem in this context is the following:

\emph{Given any $(y_0^\pm,y_1^\pm)$, does there exist a (control) function $F$ such that the solution $y$ of \eqref{eq_ctrl_sys_MB} satisfies}
\[ (y, \partial_t y) = (y_0^+,y_1^+), \quad \text{ at final time}?\]

\subsection{Literature}

First, let us give a brief review of some results on control of waves on time dependent domains. The work of \emph{Bardos-Chen} \cite{bard_chen:ctrlstab_wave3} solved the interior control problem for the free wave equation\footnote{That is, for $-\partial_{tt}^2 y + \Delta_x y =0$, without any lower order terms.} in a domain that is expanding in time. One of the geometric assumptions in this work requires that the energy of the wave equation does not change too much under the boundary expansion. This result was extended to geometric free waves in \cite{lu_li_chen_yao:ctrlstab_wave_moving}, using Riemannian geometric methods. 

Furthermore, the work of \emph{Miranda} \cite{mira:hum_var_coeff} solves a suitable boundary control problem for a hyperbolic PDE on static cylinders by using the \emph{Hilbert uniqueness method} (developed in \emph{Lions} \cite{lionj:control_hum,lionj:ctrlstab_hum}). This result is then used in \emph{Miranda} \cite{mira:control_var_boundary} to solve a boundary control problem for the free wave equation on time dependent domains that have the form
\[ \bigcup_{\tau} \left( \{ \tau \} \times k(\tau) \cdot \Omega \right), \]
for a suitable function $k(\tau)$, which ensures that the domain becomes cylindrical as $t \rightarrow \infty$.

In the one-dimensional case, some results related to control of waves on domains with moving boundaries can be found in \cite{cui_jiang_wang:control_wave_fec, sun_li_lu:control_wave_moving, sengou:obs_control_wave, sengou:obs_control_wave2, wang_he_li:control_wave_noncyl}. 
For instance, the work in \emph{Cui, Jiang, and Wang} \cite{cui_jiang_wang:control_wave_fec} considers the boundary control problem for the equation described by the motion of a string with one end point fixed, and the control acting on the fixed end point. A similar string equation control problem was also considered in \emph{Sun, Li, and Lu} \cite{sun_li_lu:control_wave_moving}, where the control was put on the moving end point.

The recent article \emph{Nakao} \cite{nakao} solves an interior control problem for the free wave equation, with initial data in $H^1_0 \times L^2$ and the added assumption that the $\Omega_\tau$ does not change \emph{too rapidly} along $\tau$. However, we consider the wave equation with lower order time dependent coefficients and obtain better control regions.

A very general result for boundary control of waves on time dependent domains is the work in \emph{Shao} \cite{Arick}. This result is applicable to the $(n+1)$-dimensional wave equation with time dependent lower order coefficients, on time dependent domains. The geometric assumptions in this work are that the boundary is timelike and any two $\Omega_{\tau_1},\Omega_{\tau_2}$ are diffeomorphic.

\subsection{Setting}

To present our main result we give some precise definitions and explain the setting of our problem.

\begin{definition} \label{def_U_domain_MB}
Let $\tau_-,\tau_+ \in \R$ be such that $\tau_- < \tau_+$. Then, we consider a domain $\mc{U}$ that satisfies the following properties:
\begin{enumerate}
\item $\mc{U} \subset \R^{1+n}$, with smooth timelike boundary $\partial \mc{U}$.

\item For each $\tau \in \R$, define $\Omega_\tau$ as
\begin{equation} \label{eq_def_om_t_MB}
\Omega_\tau := \{ x \in \R^n |(\tau,x) \in \mc{U} \}.
\end{equation}
Then each $\Omega_\tau$ is a non-empty, bounded, and open subset of $\R^n$.

\item There exists a smooth future directed timelike vector field $\mc{Z}$ on $\bar{\mc{U}}$ such that $\mc{Z}|_{\partial \mc{U}}$ is tangent to $\partial \mc{U}$.
\end{enumerate}
\end{definition}

The above definition of $\mc{U}$ implies that \( \mc{U} \) looks like
\begin{equation} \label{eq_def_U_1_MB}
\mc{U} := \bigcup_{ \tau \in \R } \Big( \{ \tau \} \times \Omega_\tau \Big) \text{.}
\end{equation}

Consider $\Omega_0$ from \eqref{eq_def_om_t_MB} and let $(x^1,\ldots,x^n)$ be a coordinate system on $\Omega_0$. This coordinate system can be moved to any $\Omega_\tau$ by using the integral curves of $\mc{Z}$; we leave the coordinates constant as we move along the integral curves. Hence, using the moving coordinates $(t,x^1,\cdots,x^n)$ gives that $\mc{U} \simeq \R \times \Omega_0$. That is, the domain $\mc{U}$ can be reparametrised into a time-static cylinder.

Now, for $\tau \in \R$, let $\mc{U}_\tau$ denote the cross section
\[ \mc{U}_\tau := \mc{U} \cap \{ t= \tau \}. \]

For given $\tau_-,\tau_+ \in \R$ satisfying $\tau_- < \tau_+$, define
\begin{align*}
\mc{U}_{\tau_-,\tau_+} & := \mc{U} \cap \{ \tau_- < t < \tau_+ \}, \numberthis \label{eq_def_U_pm_MB}\\
\partial\mc{U}_{\tau_-,\tau_+} & := \partial\mc{U} \cap \{ \tau_- < t < \tau_+ \}.
\end{align*}

In this setting, we consider the following PDE as our control system
\begin{equation} \label{eq.intro_wave_MB}
\begin{cases} - \partial_{tt}^2 y + \Delta_x y - \nabla_\mc{X} y + q y = F \textbf{1}_W, \qquad & \text{in } \mc{U}_{\tau_-,\tau_+} \text{,} \\
(y,\partial_t y) = (y_0^- ,y_1^-),& \text{on } \mc{U}_{\tau_-}, \\
 y = 0, & \text{on } \partial \mc{U}_{\tau_-,\tau_+},
\end{cases}
\end{equation}
where $(y_0^- ,y_1^-) \in L^2(\mc{U}_{\tau_-}) \times H^{-1} (\mc{U}_{\tau_-})$, $\mc{X} \in C^\infty(\bar{\mc{U}};\R^{1+n})$ is a vector field, $q \in C^\infty(\bar{\mc{U}})$ is the potential, and $W \subset \mc{U}_{\tau_-,\tau_+}$ is an interior subset. Then, we want to answer the following question: 

\begin{problem} \label{eq_ctrl_pblm_MB}
Given any $(y_0^{\pm} ,y_1^{\pm}) \in L^2(\mc{U}_{\tau_{\pm}}) \times H^{-1} (\mc{U}_{\tau_{\pm}})$, does there exist a function $F$, such that the solution of \eqref{eq.intro_wave_MB} satisfies
\begin{equation}
(y,\partial_t y) = (y_0^+ ,y_1^+), \quad \text{on } \mc{U}_{\tau_+}?
\end{equation}
\end{problem}

Due to a standard duality argument (see \cite{dolec_russe:obs_control, lionj:control_hum}), to solve the control problem we want to establish an \emph{observability inequality} of the type
\begin{equation} \label{eq_intro_obs_ineq_MB}
\int_{\mc{U}_{\tau_\pm}} (|\nb_{t,x} \phi|^2 + \phi^2) \lesssim \int_W (|\partial_t \phi|^2 + \phi^2),
\end{equation}
where \(\phi\) is the solution of the adjoint system
\begin{equation}\label{eq.intro_obs_MB}
\begin{cases}
-\partial^2_{tt} \phi + \Delta_x \phi + \nabla_\mathcal{X} \phi + V \phi= 0 , \qquad & \text{in } \mc{U}_{\tau_-,\tau_+},\\
( \phi, \partial_t\phi)= ( \phi_0, \phi_1), \qquad & \text{on } \mc{U}_{\tau_-}, \\
\phi = 0, & \text{on }  \partial \mc{U}_{\tau_-,\tau_+},
\end{cases}
\end{equation}
with $ ( \phi_0, \phi_1) \in H^1_0(\mc{U}_{\tau_-}) \times L^2(\mc{U}_{\tau_-})$, and $V= q + \nb_\alpha \mc{X}^\alpha$. Henceforth, our goal will be to show the above observability inequality.

There are various methods that can be used to prove observability, such as multiplier methods, microlocal analysis, and Carleman estimates. Out of the three, Carleman estimates are applicable to a bigger class of PDEs, in that they can be applied to waves with lower order coefficients and domains that are time dependent. Hence, we will derive a suitable Carleman estimate to solve the considered controllability/observability problem.

Carleman estimates are weighted integral estimates which can be used for showing suitable unique continuation properties for PDEs, see the work by \emph{Carleman} \cite{carl:uc_strong}. In our context, modern use of Carleman estimates derive from the work of \emph{H\"ormander} \cite{hor:lpdo4} (also see \emph{Cald\'eron} \cite{cald:unique_cauchy}). 

In the case of time static domains, for interior control results obtained via Carleman estimates, one has the control region given by $W := (-T,T) \times \omega$, where
\begin{align*}
\omega := \mc{O}_\sigma (\Gamma_+) \cap \Omega, \qquad & \Gamma_+ := \{ x \in \partial \Omega | (x-x_0)\cdot \nu >0 \}, \numberthis \label{eq_std_carl_rslt_MB}\\
\mc{O}_\sigma (\Gamma_+) := \{ y \in \R^n : & |y-x| < \sigma, \text{ for some } x \in \Gamma_+ \},
\end{align*}
where $x_0 \notin \bar{\Omega}$ and $\nu$ is the outward unit normal of $\Omega$. For instance, see \cite{Carleman, FYZ, Zhang}.

\iffalse
These estimates have a general structure that roughly looks like
\begin{equation} 
|| \upzeta \nb_{t,x} \phi||_\mc{H}^2 + || \upzeta \phi||_\mc{H}^2 \lesssim \frac{1}{a} || \upzeta (-\partial_{tt}^2 + \Delta_x) \phi ||_\mc{H}^2 + \ldots,
\end{equation}
where $\mc{H}$ is a suitable Hilbert space depending on the problem, $a\gg1$ is a free parameter, and $\upzeta$ is the Carleman weight function that depends on $a$.

For the purpose of solving observability estimates, we derive an estimate of the type
\begin{equation} \label{eq_carl_gnrl}
||\upzeta \nb_{t,x} \phi||_{L^2(\Omega)}^2 + ||\upzeta \phi||_{L^2(\Omega)}^2 \lesssim \frac{1}{a} ||\upzeta (-\partial_{tt}^2 + \Delta_x) \phi ||_{L^2(\Omega)}^2 + a || \upzeta \phi ||_{L^2(W)}^2.
\end{equation} 
Even without going into the technicalities of the proof, straightforward comparison of \eqref{eq_intro_obs_ineq_MB} and \eqref{eq_carl_gnrl} shows the reason for using Carleman estimates here. Applying suitable energy estimates to \eqref{eq_carl_gnrl} for taking care of the weight $\upzeta$, results in \eqref{eq_intro_obs_ineq_MB}. While proving observability, the free parameter $a$ is used to absorb any bad terms arising from the lower order terms in the PDE \eqref{eq.intro_obs_MB}.
\fi

\subsection{A regularity issue}
The article \cite{jena} by the author, solves the interior control problem for wave equations with time dependent lower order terms on \emph{time-static} domains and obtains significantly improved control regions compared to standard Carleman based results. This is done by obtaining a Carleman estimate for \emph{wave-type} operators in the setting of $\R_t^m \times \R_x^n$; such operators are also known as ultrahyperbolic operators. Note that this work considers wave equation with initial data in the regular space $H^1_0 \times L^2$. However, in the current setting of time dependent domains, we consider wave equation with initial data in the weaker space $L^2 \times H^{-1}$.

On time dependent domains the control problem for waves with initial data in $H^1_0 \times L^2$ is not yet solved. There is a regularity issue which prevents us from considering initial data in the regular space $H^1_0 \times L^2$. In this case, due to duality, the observability problem has initial data in $L^2 \times H^{-1}$, and it is not possible to get a $L^2$-Carleman estimate for $\phi$ using the same method as in \cite{jena}. Indeed, a crucial step in the analysis in \cite{jena} involves defining a new function $z$ as follows
\begin{align*}
z(t_1,t_2,x) := \int_{t_1}^{t_2} \phi(s, x) ds,
\end{align*}
which allows one to bypass the regularity issue because $z$ has better regularity than $\phi$. Then one uses a suitable $H^1$ Carleman estimate for $z$, which is easier to obtain than directly deriving a $L^2$ estimate for $\phi$. Transforming $z$ back to $\phi$ leads to the required observability inequality for $\phi$. However, for time dependent domains the function \(z\) now turns out to be
\begin{align*}
z(t_1,t_2,x) := \int_{t_1}^{t_2} \phi(s, x(s)) ds,
\end{align*}
where the \(x(s)\) inside the integral represents the fact that the domain is now time dependent. This $x(s)$ factor creates several new terms in the \emph{wave-type} equation for $z$, when we differentiate $z$. These terms cannot be controlled properly as the expressions become too complicated, and it causes serious issues in the required analysis. This seems to suggest that this might not be the optimal method for solving the observability problem in the time dependent case. To show observability for this problem, we need more knowledge of wave equations in general geometries. 

However, on time dependent domains, observability can be proved for the adjoint system with initial data in $H^1_0 \times L^2$, by directly using a suitable $H^1$-Carleman estimate. That is, the control problem for waves with initial data in $L^2\times H^{-1}$ is solvable, which is what we consider in this article. Roughly speaking, the analysis in the current work is a simpler version of the proof of \cite{jena}, but applied to the more general setting of time dependent domains.

\subsection{Main result}
To present the main result of this chapter, we give the following definition.
\begin{definition}
For given $(t_0, x_0) \in \R^{1+n}$, define $f_0$ and $\mc{D}_0$ as follows
\begin{equation} \label{eq_def_f0_MB}
f_0 := \frac{1}{4} [ |x-x_0|^2 - (t-t_0)^2 ], \qquad \mc{D}_0 := \{ f_0 > 0 \}.
\end{equation}
\end{definition}
The $f_0$ defined above can be seen as a time dependent version of the usual distance functions in standard Carleman results; see \cite{Carleman, Zhang}. Here, $\mc{D}_0$ is the exterior of the null cone centred at $(t_0,x_0)$.

\begin{definition}
Define the set $\Gamma'$ 
\begin{equation} \label{eq_gamma'_MB} 
\Gamma' := \partial\mc{U}_{\tau_-,\tau_+} \cap \mc{D}_0 \cap \{ \mc{N} f_0 > 0 \},
\end{equation}
where $\mc{N}$ denotes the Minkowski outward unit normal of $\mc{U}$.
Then, using \eqref{eq_def_U_1_MB} we get the following
\begin{equation}
\Gamma' = \left[ \bigcup_{\tau \in (\tau_-,\tau_+)} \left(\{\tau\} \times \partial \Omega_\tau \right) \right] \cap \mc{D}_0 \cap \{ \mc{N} f_0 > 0 \}.
\end{equation}
Now if $(\tau,y) \in \Gamma'$, then
\begin{align*}
(\tau,y) \in \{ \tau \} \times \partial \Omega_\tau, \qquad f_0(\tau,y) > 0, \qquad \mc{N}f_0(\tau,y)>0.
\end{align*}
Then for $\sigma>0$, we define the following $\sigma$-neighbourhoods
\begin{align} \label{eq_def_Osig_gam_MB}
\mc{O}_\sigma(y) & := \{ y_1 \in \R^n : |y_1-y| < \sigma\} \subset \R^n, \\
\mc{O}_\sigma(\Gamma') & := \bigcup_{(\tau,y) \in \Gamma'} \Big(  \{ \tau \} \times \mc{O}_\sigma(y) \Big) \subset \R^{1+n}. \notag
\end{align}
Let $W'$ be a neighbourhood of $\overline{\mc{O}_\sigma(\Gamma')}$ in $\mc{U}$, that is, $W' \subset \mc{U}$ such that
\begin{equation} \label{eq_W'_MB}
W' \supset \overline{\mc{O}_\sigma(\Gamma')} \cap \mc{U}.
\end{equation}
\end{definition}
Due to Definition \ref{def_U_domain_MB}, we see that $\mc{O}_\sigma(\Gamma')$ varies smoothly along $\tau$. Moreover, note the difference between \eqref{eq_std_carl_rslt_MB} and \eqref{eq_gamma'_MB}. It is precisely this restriction of the considered boundary region to $\mc{D}_0$ that gives us improved control regions. Then, our main observability result is as follows.

\begin{theorem} \label{thm_obs_main_intro_MB}
Let $x_0 \in \R^n$ be fixed. Let $\tau_\pm \in \R$ be such that 
\begin{equation} \label{eq_tt_RR_MB}
\tau_+ - \tau_- > R_+ + R_-, \qquad R_\pm := \sup_{(\tau_\pm,y) \in \partial\mc{U}} |y-x_0|.
\end{equation}
Let $t_0 \in (\tau_-,\tau_+)$ be chosen such that
\begin{equation} \label{eq_t0_MB}
t_0 - \tau_- > R_-, \qquad \tau_+ - t_0 > R_+.
\end{equation}
Furthermore, let $\Gamma'$ and $W'$ be defined as in \eqref{eq_gamma'_MB} and \eqref{eq_W'_MB}, respectively. Then, for some $C>0$ the following is satisfied
\begin{equation} \label{eq_obs_main_0_MB}
\int_{\mc{U} \cap \{t=\tau_\pm\}} ( |\nb_{t,x} \phi|^2 + \phi^2 ) \leqslant C \int_{W'} ( |\partial_t \phi|^2 + \phi^2) \text{,}
\end{equation}
for any solution $\phi \in C^2(\mc{U}) \cap C^1(\bar{\mc{U}})$ of \eqref{eq.intro_obs_MB} satisfying $\phi|_{\partial \mc{U}_{\tau_-,\tau_+} \cap \mc{D}_0}=0$.
\end{theorem}

\begin{remark}
Since wave equation has finite speed of propagation, one must wait sufficient time so that effects of the control can reach all the points in the domain. Assumption \eqref{eq_tt_RR_MB} takes care of this fact.
\end{remark}

\begin{remark}
Note that \eqref{eq_tt_RR_MB} implies the existence of a $t_0$ satisfying \eqref{eq_t0_MB}. Moreover, the regions $\Gamma'$ and $W'$ depend on $f_0$, which in turn depends on $t_0$.
\end{remark}

\begin{remark} \label{ch4_rmk_general_xt}
A crucial point here is that in the case of static domains (see \eqref{eq_std_carl_rslt_MB}) we have the requirement $(x-x_0)\cdot \nu >0$. But this is now replaced with the condition $\mc{N} f_0 >0$, or equivalently
\[ (x-x_0)\cdot\nu - (t-t_0)\cdot \nu^t >0, \]
where $\mc{N} := (\nu^t,\nu)$. Thus, $\mc{N} f_0 >0$ can be seen as a time dependent domain analogue of the condition $(x-x_0)\cdot \nu >0$. 

Moreover, in the time static case, for the observability estimate one is only concerned with the observation point $x_0$ in the spatial domain $\R^n$. However, due to the time dependent nature of the geometry, now we consider the observation point to be $(t_0,x_0)$ in the spacetime domain $\R^{1+n}$. Furthermore, it is exactly because of this reason that we now work with $f_0$ defined as in \eqref{eq_def_f0_MB}. For a detailed discussion on this see \cite{Arick}.
\end{remark}

\begin{remark}
In the above result we assumed that the coefficients $V, \mc{X}$ are smooth. However, the regularity of the coefficients can be lowered using standard analytic arguments. As regularity is not a primary interest in this article, we avoid this discussion here.
\end{remark}

Now, restricting the observation region close to the exterior region $\mc{D}_0$ is a major feature of our result. In the static case (see \eqref{eq_std_carl_rslt_MB}), the corresponding set $\Gamma_+$ is only space dependent and taking a $\sigma$-neighbourhood in the spatial direction is enough to obtain observability. However, in the current time dependent setting, we have to be careful while taking the $\sigma$-neighbourhood because $\Gamma'$ from \eqref{eq_gamma'_MB} is spacetime dependent. For this purpose, we consider each time-slice of $\Gamma'$ and then take spatial neighbourhoods of each of these slices. Taking union of these individual neighbourhoods over $\tau$, gives us the required region $\mc{O}_\sigma(\Gamma')$ as in \eqref{eq_def_Osig_gam_MB}. Now, although our actual observation region is given by $W'$ from \eqref{eq_W'_MB}, the difference between $W'$ and $\mc{O}_{\sigma,f} (\Gamma') \cap (\mc{U}_{\tau_-,\tau_+} \cap \mc{D}_0)$ can be made arbitrarily small. This allows us to obtain the best known Carleman type result for interior control of waves on time dependent domains. Using standard Carleman methods, one would expect the control region to be $\mc{O}_{\sigma,f} (\Gamma') \cap (\mc{U}_{\tau_-,\tau_+})$.

Finally, using the standard duality argument, a corollary to Theorem \ref{thm_obs_main_intro_MB} is the following controllability result.
\begin{theorem} \label{thm_control_MB}
Consider the system \eqref{eq.intro_wave_MB}. Assume that the hypothesis of Theorem \ref{thm_obs_main_intro_MB} holds and define $W' \subset \mc{U}_{\tau_-,\tau_+}$ to be any open set satisfying 
\begin{equation}
W' \supset \overline{\mc{O}_\sigma (\Gamma') \cap (\mc{U}_{\tau_-,\tau_+} \cap \mc{D}_0)}.
\end{equation}
Then, given any $(y_0^\pm,y_1^\pm) \in L^2(\mc{U}_{\tau_\pm}) \times H^{-1} (\mc{U}_{\tau_\pm})$, there exists a control function\footnote{Here $(H^1(W'))^*$ denotes the dual space of $H^1(W')$.} $ F \in (H^1(W'))^*$ such that the solution $y$ of \eqref{eq.intro_wave_MB} satisfies
\[ (y,\partial_t y) = (y_0^+ ,y_1^+), \qquad \text{on } \mc{U}_{\tau_+}. \]
\end{theorem}

\subsection{Key features} \label{ssec_key_feat_MB}
The result presented in this article has the following features:
\begin{enumerate}

\item We obtain a novel control result for wave equations with time dependent lower order coefficients, on time dependent domains. We do not assume any time analyticity for these coefficients.

\item Our control region is significantly improved compared to standard Carleman based results. The control region we obtain is restricted to a neighbourhood of $\mc{O}_\sigma (\Gamma') \cap (\mc{U}_{\tau_-,\tau_+} \cap \mc{D}_0)$. Usual Carleman methods would lead to control regions of the type $\mc{O}_\sigma (\Gamma') \cap \mc{U}_{\tau_-,\tau_+}$.

\item We show controllability when the point $(t_0,x_0) \in \R^{1+n}$. Usually Carleman results require that $(t_0,x_0) \notin \bar{\mc{U}}$ (see for example, \cite{Carleman, FYZ, lasie_trigg_zhang:wave_global_uc, zhang:obs_wave_lower} where time independent domains are considered).

\end{enumerate}
Feature (3) is also present in---\cite{Arick} for time dependent case and \cite{fu_liao, jena, LZZ} for time static case.

\subsection{Outline}
The rest of the chapter is divided as follows.
\begin{itemize}

\item In Section \ref{sec_geo_set}, we present the details of the geometric set-up for the control problem.

\item In Section \ref{ch4_sec_int_carl}, we prove an interior Carleman estimate. For this purpose, we use the boundary Carleman estimate result of \cite{Arick}. We estimate the boundary integral term by an interior integral term using a suitable vector field and some cut-off functions.

\item In Section \ref{ch4_sec_obs}, we use the interior Carleman estimate along with certain energy results to complete the proof. We also adapt some techniques from \cite{Arick}, where the author proves boundary controllability of waves on similar time-dependent domains. The major difference is that now we achieve observability with an interior term in the Carleman estimate, instead of a boundary term.
\end{itemize}

\section{Geometric setting} \label{sec_geo_set}

In this section, we describe the geometric background for the results of this article.

\begin{definition}\label{def_setting}
Let \(n \in \N\) be fixed. On \( \R^{1+n} \), we define the following
\begin{itemize}

\item Denote by \( t \) and \( x = (x_1, \ldots,  x_n)\) the \textbf{Cartesian coordinates} on \( \R^{1+n} \). Here, \(t\) maps to the first component of \(\R^{1+n}\) and \(x\) maps to the remaining \(n\) components of \(\R^{1+n}\).

\item Let \(g\) denote the Minkowski metric on \(\R^{1+n}\), defined as 
\begin{equation}\label{eq.gmetric}
g = -dt^2 + dx_1^2 + \cdots + dx_n^2\text{.}
\end{equation}

\item Let \(r := |x| \) denote the \textbf{spatial radial function}, that is, we have
\[ r := \sqrt{(x_1)^2+ \ldots + (x_n)^2}. \] 

\item We define the \textbf{null coordinates} as follows
\begin{equation} \label{eq.uv}
u:= \frac{1}{2}(t - r) \text{,} \qquad v:= \frac{1}{2} (t + r) \text{.}
\end{equation}
The inverse mapping is given by
\begin{equation}\label{eq.rt}
t = v + u \text{,} \qquad r = v - u \text{.}
\end{equation}

\item Now, define the function \(f\) as
\begin{equation}\label{eq.f}
 f := -uv = \frac{1}{4}(r^2 - t^2) = \frac{1}{4} \left(|x|^2-t^2\right) \text{.}
\end{equation}
Note that, the level sets for each value of $f$ look like-
\begin{itemize}
\item for $\{f<0\}$, they are two sheeted hyperboloids on $\R^{1+n}$,
\item for $\{f=0\}$, it forms the null cone centred at the origin,
\item for $\{f>0\}$, they are one sheeted hyperboloids on $\R^{1+n}$.
\end{itemize}
\end{itemize}

\end{definition}

The function \(f\) is one of the crucial requirements for obtaining the Carleman estimate. In particular, it is used to define the Carleman weight. It also helps us to localise the Carleman estimate to the exterior of the null cone.

We also present some equivalent representations of \(g\) in other coordinate systems. Let \(\mathbb{S}^{n-1}\) denote the unit sphere in \(\R^n\), that is 
\[ \mathbb{S}^{n-1} = \{ y \in \R^n : |y| =1\}. \]
Then, on \( \{r \neq 0\}\) standard polar coordinates are given by \( (r,t,\omega) \), where \(\omega\) is the angular coordinate with values in \(\mathbb{S}^{n-1}\). And, the null coordinates are given by \(( u, v, \omega )\). The coordinate vector fields with respect to these coordinate systems are denoted by \( \partial_t, \partial_r, \partial_u, \partial_v \). Next, let \(\mathring{\gamma}_{\mathbb{S}^{n-1}}\) denote the unit round metric on \(\mathbb{S}^{n-1}\).

We also define the shifted coordinate system with respect to a fixed point $p \in \R^{1+n}$.
\begin{definition}
Let $p \in \R^{1+n}$ be fixed. Let $t(p)$ and $x(p)$ denote the coordinates of $p$ with respect to the usual Cartesian coordinate system centred at $(0,0) \in \R^{1+n}$. Then, we define the following notations
\begin{align*}
t_p := t - t(p), & \qquad x_p := x - x(p), \\
r_p := |x_p|, \qquad u_p := \frac{1}{2} (t_p - r_p), & \qquad v_p := \frac{1}{2} (t_p + r_p), \qquad f_p := - u_p v_p, \\
\end{align*}
\end{definition}

Then, we can write \(g\) as follows
\begin{equation}\label{eq.psr_met}
g = -dt_p^2 + dr_p^2 + r_p^2 \mathring{\gamma}_{\mathbb{S}^{n-1}} = -4 du_p dv_p + r_p^2 \mathring{\gamma} \text{,}
\end{equation}
where $dt_p,\ dr_p,\ du_p,\ dv_p$ are defined with respect to the shifted coordinate system centred at $p$.

There are a few coordinate systems at work here. Hence, we define some convention to specify the coordinates we use at any step.
\begin{definition}\label{def_coord}
Let us define the following convention:
\begin{itemize}
\item \( (\alpha, \beta, \ldots) :\) Lower case Greek letters, ranging from 0 to \(n\), denote space-time components in \(\R^{1+n}\).
\item \( (a,b,\ldots) :\) Lower case Latin letters, ranging from 1 to \(n-1\), denote angular components corresponding to \(\omega_x \in \mathbb{S}^{n-1} \) in the above mentioned coordinate systems.
\end{itemize}
\end{definition}

\begin{definition}\label{def.covder}
We define the following notations for operators on \( (\R^{1+n},g) \):
\begin{itemize}
\item \( \nabla \) denotes the Levi-Civita connection with respect to \(g\).
\item \(  \square  := g^{\alpha\beta}  \nabla_{\alpha\beta} \) denotes the wave operator with respect to  \(g\).
\item \(  \nasla  \) denotes the derivatives in the angular components with respect to \(g\).
\end{itemize}
\end{definition}

\begin{definition}
Define \( \mc{D} \subset \R^{1+n} \) as follows
\begin{equation}\label{eq.D}
\mc{D} := \{ f > 0 \} \text{.}
\end{equation}
\end{definition}
Note that, $\mc{D}$ is the region exterior to the null cone centred at the origin; see the discussion following equation \eqref{eq.f}. Analogously, we also define the region $\mc{D}_p$ as follows
\begin{equation}\label{eq.D_p}
\mc{D}_p := \{ f_p > 0 \} \subset \R^{1+n} \text{.}
\end{equation}
The region $\mc{D}_p$ plays a crucial role in the main Carleman estimate, as well as the subsequent observability result.

\section{Carleman estimate} \label{ch4_sec_int_carl}
We will now state the main Carleman estimate of \cite{Arick}, which obtained a boundary control result for wave equations on time dependent domains; we refer the reader to the corresponding article for the proof.

\begin{theorem}[Boundary Estimate]\label{thm.carl_bdry_MB}
Let $\mc{U}$ be defined as in Definition \ref{def_U_domain_MB}. Fix $p \in \R^{1+n}$. Now, assume that there exists \( R>0 \) such that
\begin{equation}\label{eq.carleman_domain_MB}
\mc{U} \cap \mc{D}_p \subseteq \{ r_p<R \}.
\end{equation}
Choose constants \( \varepsilon, a, b >0 \) satisfying
\begin{equation}\label{eq.carleman_choices_MB}
a \geqslant n^2, \qquad \varepsilon \ll_n b \ll R^{-1}\text{.}
\end{equation}
Then, there exists a constant \( C >0 \), such that for any \(\phi \in C^2({\mc{U}}) \cap C^1(\bar{\mc{U}}) \) satisfying
\begin{equation}\label{eq.carleman_dirichlet_MB}
\phi |_{\partial\mc{U} \cap \mc{D}_p} = 0 \text{,}
\end{equation}
we have the following estimate
\begin{align}\label{eq.carleman_est_MB}
\notag & C \varepsilon \int_{\mc{U} \cap \mc{D}_p} \zeta_{a,b;\varepsilon}^p r_p^{-1}(|u_p \partial_{u_p} \phi|^2 + |v_p \partial_{v_p} \phi|^2 + f_p g^{ab}\slashed\nabla_a^p \phi \slashed\nabla_b^p \phi ) + Cba^2\int_{\mc{U}\cap \mc{D}_p}\zeta_{a,b;\varepsilon}^p f_p^{-\frac{1}{2}} \phi^2 \\
& \qquad \leqslant \frac{1}{a}\int_{\mc{U}\cap \mc{D}_p} \zeta_{a,b;\varepsilon}^p f_p |\square \phi|^2 +  C' \int_{\partial\mc{U}\cap \mc{D}_p}\zeta_{a,b;\varepsilon}^p [( 1 - \varepsilon r_p ) \mc{N} f_p + \varepsilon f_p \mc{N} r_p ] |\mathcal{N} \phi|^2 \text{,} 
\end{align}
where \( \zeta_{a,b;\varepsilon}^p \) is defined as 
\begin{equation}
\label{eq.carleman_weight_MB} \zeta_{ a, b; \varepsilon }^p := \left\{ \frac{ f_p }{ ( 1 + \varepsilon u_p ) ( 1 - \varepsilon v_p ) } \cdot \exp \left[ \frac{ 2 b f_p^\frac{1}{2} }{ ( 1 + \varepsilon u_p )^\frac{1}{2} ( 1 - \varepsilon v_p )^\frac{1}{2} } \right] \right\}^{2a} \text{,}
\end{equation}
and \(\mc{N}\) is the outward pointing unit normal of \(\mc{U}\) with respect to \(g\).
\end{theorem}

\begin{remark}
The work in \cite{jena} generalised the above result to obtain a Carleman estimate for ultrahyperbolic operators, that is, the case when $(t,x) \in \R_t^m \times \R_x^n$.
\end{remark}
Since we are looking to solve the interior control problem, we need to derive an interior Carleman estimate. For this purpose, we give the following definition.
\begin{definition} \label{def_Gamma_MB}
Let us define $\Gamma_+$ as follows
\begin{equation} \label{eq_def_Gamma_MB}
\Gamma_+  := \partial \mc{U} \cap \mc{D}_p \cap \{ ( 1 - \varepsilon r_p ) \mc{N} f_p + \varepsilon f_p \mc{N} r_p > 0 \} \text{.}
\end{equation}
\end{definition}

Now we define $\mc{O}_\sigma(\Gamma_+)$ using the same idea we used for defining $\mc{O}_\sigma(\Gamma')$ in \eqref{eq_def_Osig_gam_MB}.
\begin{definition} \label{def_omega_MB}
Let $\sigma>0$. Then, define the following $\sigma$-neighbourhoods
\begin{align} \label{eq_def_W1_MB}
\mc{O}_\sigma(y) & := \{ y_1 \in \R^n : |y_1-y| < \sigma\} \subset \R^n, \\
\mc{O}_\sigma(\Gamma_+) & := \bigcup_{(\tau,y) \in \Gamma_+} \Big(  \{ \tau \} \times \mc{O}_\sigma(y) \Big) \subset \R^{1+n}. \notag
\end{align}
Now, let $W \subset \mc{U} \cap \mc{D}_p$ be an open set defined as follows
\begin{equation} \label{eq_def_W2_MB}
W = \mc{O}_\sigma (\Gamma_+) \cap (\mc{U} \cap \mc{D}_p).
\end{equation}
\end{definition}

\begin{remark}
The $W$ defined above can be written as 
\begin{equation}
W = \bigcup_{\tau \in (\tau_-,\tau_+)} \left( \{ \tau \} \times \omega_\tau \right) \text{,}
\end{equation}
for some $\omega_\tau \subset \Omega_\tau $, for $\tau \in (\tau_-,\tau_+)$.
\end{remark}

With the above definitions, we have the following interior Carleman estimate.
\begin{theorem}[Interior Estimate] \label{thm.carl_int_MB}
Let \( \mc{U} \) be a defined as in Definition \ref{def_U_domain_MB}.
Fix \( R>0 \) such that it satisfies
\begin{equation}\label{eq.carl_int_domain_MB}
\bar{\mc{U}} \cap \mc{D}_p \subseteq \{ r_p < R \} \text{.}
\end{equation}
Choose constants \( \varepsilon, a, b >0 \) such that they satisfy
\begin{equation}\label{eq.carl_int_choices_MB}
a \geqslant n^2 \text{,} \qquad a \gg R \text{,} \qquad \varepsilon \ll_n b \ll R^{-1} \text{.}
\end{equation}
Furthermore, let $W$ be defined as in Definition \ref{def_omega_MB}.
Then, there exists a constant \(C>0\) such that for any \( \phi \in C^2({\mc{U}})\cap C^1(\bar{\mc{U}}) \) satisfying \(\phi|_{\partial\mc{U} \cap \mc{D}_p} = 0,\) the following estimate holds
\begin{align*}
C \varepsilon & \int_{\mc{U} \cap \mc{D}_p} \zeta_{a,b;\varepsilon}^p r_p^{-1}(|u_p \partial_{u_p} \phi |^2 + |v_p \partial_{v_p} \phi |^2 + f_p g^{ab}\slashed\nabla^p_a \phi \slashed\nabla^p_b \phi ) + Cba^2\int_{\mc{U} \cap \mc{D}_p}\zeta_{a,b;\varepsilon}^p f_p^{-\frac{1}{2}} \phi^2 \\
& \quad \leqslant \frac{1}{a} \int_{\mc{U}\cap \mc{D}_p} \zeta_{a,b;\varepsilon}^p f_p |\square \phi|^2 + a R^2 \int_W \zeta_{a,b;\varepsilon}^p f_p^{-1} |\partial_t \phi|^2 + a^4 R^4 \int_W \zeta_{a,b;\varepsilon}^p f_p^{-3} \phi^2 \text{.} \numberthis \label{eq.carl_int_MB}
\end{align*}
\end{theorem}

The proof of the above result is in the same spirit as the proof of the Carleman estimate presented in \cite[Theorem 3.27]{jena}, which is applicable to domains that are static in time. However, due to the time dependent domains considered in the current result, some new subtleties arise here. We also provide the following result from \cite[Lemma 3.29]{jena} that will be used later; we omit the proof here.

\begin{lemma} \label{thm.der_carl_est}
Assume the hypothesis of Theorem \ref{thm.carl_int_MB}. In the region \(\mc{U}\cap \mc{D}\) we have the following estimate
\begin{equation}\label{prf.der_zeta}
|\nabla^\alpha\zeta_{a,b;\varepsilon}| \lesssim a R \zeta_{a,b;\varepsilon} f^{-1},
\end{equation}
where \(\alpha\) represents derivatives in the Cartesian coordinates.
\end{lemma}

\begin{proof}[Proof of Theorem \ref{thm.carl_int_MB}]
Assume that the hypothesis of Theorem \ref{thm.carl_int_MB} holds. That is, we have \( \mc{U}, R, a, b, \varepsilon, W, \phi \) as in the statement of the theorem. Then, this choice of the domain and the constants also satisfies the hypothesis of Theorem \ref{thm.carl_bdry_MB}. Hence, we can apply Theorem \ref{thm.carl_bdry_MB} to get
\begin{align*}
C \varepsilon & \int_{\mc{U} \cap \mc{D}_p} \zeta_{a,b;\varepsilon}^p r_p^{-1}(|u_p \partial_{u_p} \phi|^2 + |v_p \partial_{v_p} \phi|^2 + f_p g^{ab}\slashed\nabla_a^p \phi \slashed\nabla_b^p \phi ) + Cba^2\int_{\mc{U}\cap \mc{D}_p}\zeta_{a,b;\varepsilon}^p f_p^{-\frac{1}{2}} \phi^2 \\
& \qquad \leqslant \frac{1}{a}\int_{\mc{U}\cap \mc{D}_p} \zeta_{a,b;\varepsilon}^p f_p |\square \phi|^2 +  C' \int_{\partial\mc{U}\cap \mc{D}_p}\zeta_{a,b;\varepsilon}^p [( 1 - \varepsilon r_p ) \mc{N} f_p + \varepsilon f_p \mc{N} r_p ] |\mathcal{N} \phi|^2 \text{.} 
\end{align*}
The boundary integral term in the above estimate can be bound as follows
\begin{equation} \label{eq.R_factor_MB}
\int_{\partial\mc{U}\cap \mc{D}_p}\zeta_{a,b;\varepsilon}^p [( 1 - \varepsilon r_p ) \mc{N} f_p + \varepsilon f_p \mc{N} r_p ] |\mc{N} \phi|^2 \lesssim R \int_{\Gamma_+}\zeta_{a,b;\varepsilon}^p |\mc{N} \phi|^2 \text{,} 
\end{equation}
where we used \eqref{eq_def_Gamma_MB} to obtain the integral region on the RHS. To prove the theorem, we only need to bound the term present in the RHS of the above expression, by an integral over the region $W$.

First, we define a vector field \( h \in C^1(\Bar{\mc{U}};\R^{1+n}) \) such that \(h = \mc{N} \) on \(\partial\mc{U}\cap \mc{D}_p\). Furthermore, also define a cut-off function \( \rho\in C^2(\Bar{\mc{U}};[0,1]) \) as follows
\begin{equation} \label{eq:3.1_MB}
\rho(t,x)
= \begin{cases}
1 \text{,}\hspace{1cm} (t,x) \in \mathcal{O}_{\sigma/3}(\Gamma_+)\cap \mc{U} \text{,}\\
0 \text{,} \hspace{1cm} (t,x) \in \mc{U} \setminus \mathcal{O}_{\sigma/2}(\Gamma_+) \text{.}
\end{cases}
\end{equation} 
We apply integration by parts to get the following
\begin{align} \label{eq_g10_MB}
\int_{\mc{U}\cap \mc{D}_p} \square \phi \rho \zeta_{a,b;\varepsilon}^p h \phi = - \int_{\mc{U}\cap \mc{D}_p}\nabla_\alpha \phi \nabla^\alpha (\rho \zeta_{a,b;\varepsilon}^p h \phi) + \int_{\partial\mc{U}\cap \mc{D}_p}\mathcal{N} \phi \rho \zeta_{a,b;\varepsilon}^p h \phi \text{,}
\end{align}
where the boundary integral term coming from \( \mc{U} \cap \partial\mc{D}_p \) vanishes because \( \zeta_{a,b;\varepsilon}^p|_{\partial \mc{D}_p} = 0 \). 
Note that, another application of integration by parts, along with using $\phi|_{\partial \mc{U} \cap \mc{D}_p} =0$, gives us 
0\begin{align*}
& \int_{\mc{U}\cap \mc{D}} \nabla_\alpha \phi \nabla^\alpha ( \rho \zeta_{a,b;\varepsilon}^p h \phi) \\
& =  \int_{\mc{U} \cap \mc{D}} \rho \zeta_{a,b;\varepsilon}^p \nabla_\alpha \phi \nabla^\alpha(h^\beta \nabla_\beta \phi) + \int_{\mc{U} \cap \mc{D}} \nabla_\alpha \phi \nabla^\alpha(\rho \zeta_{a,b;\varepsilon}^p) \cdot h \phi \\
& = \int_{\mc{U} \cap \mc{D}} \rho \zeta_{a,b;\varepsilon}^p \nabla_\alpha \phi \nabla^\alpha h^\beta \nabla_\beta \phi + \int_{\mc{U} \cap \mc{D}} \rho \zeta_{a,b;\varepsilon}^p \nabla_\alpha \phi \cdot h^\beta {\nabla^\alpha}_\beta \phi + \int_{\mc{U} \cap \mc{D}} \nabla_\alpha \phi \nabla^\alpha(\rho \zeta_{a,b;\varepsilon}^p) \cdot h \phi \\
& = \int_{\mc{U} \cap \mc{D}} \rho \zeta_{a,b;\varepsilon}^p \nabla^\alpha h^\beta \nabla_\alpha \phi \nabla_\beta \phi + \frac{1}{2} \int_{\mc{U} \cap \mc{D}} \rho \zeta_{a,b;\varepsilon}^p h^\beta \nabla_\beta (\nabla_\alpha \phi \nabla^\alpha \phi) + \int_{\mc{U} \cap \mc{D}} \nabla_\alpha \phi \nabla^\alpha(\rho \zeta_{a,b;\varepsilon}^p) \cdot h \phi \\
& =  \int_{\mc{U} \cap \mc{D}} \rho \zeta_{a,b;\varepsilon}^p \nabla^\alpha h^\beta \nabla_\alpha \phi \nabla_\beta \phi - \frac{1}{2} \int_{\mc{U} \cap \mc{D}} \nabla_\beta (\rho \zeta_{a,b;\varepsilon}^p h^\beta ) \nabla_\alpha \phi \nabla^\alpha \phi + \frac{1}{2} \int_{\partial \mc{U} \cap \mc{D}} \rho \zeta_{a,b;\varepsilon}^p |\mathcal{N} \phi|^2 \\
& \qquad  + \int_{\mc{U} \cap \mc{D}} \nabla_\alpha \phi \nabla^\alpha(\rho \zeta_{a,b;\varepsilon}^p) \cdot h \phi \text{.}
\end{align*}

Combining \eqref{eq_g10_MB} with the above expression, and also multiplying with the coefficient $R$, shows that
\begin{align*}
\frac{R}{2} \int_{\partial\mc{U}\cap \mc{D}_p}\rho\zeta_{a,b;\varepsilon}^p |\mathcal{N} \phi|^2 & = R \int_{\mc{U}\cap \mc{D}_p}\square \phi \rho\zeta_{a,b;\varepsilon}^p h \phi + R \int_{\mc{U} \cap \mc{D}_p} \rho \zeta_{a,b;\varepsilon}^p \nabla^\alpha h^\beta \nabla_\alpha \phi \nabla_\beta \phi \\
& \ \ \ - \frac{R}{2} \int_{\mc{U} \cap \mc{D}_p} \nabla_\beta (\rho \zeta_{a,b;\varepsilon}^p h^\beta ) \nabla_\alpha \phi \nabla^\alpha \phi + R \int_{\mc{U} \cap \mc{D}_p} \nabla_\alpha \phi \nabla^\alpha(\rho \zeta_{a,b;\varepsilon}^p) h\phi \text{.}
\end{align*}

Using estimates for derivatives of $\zeta_{a,b;\varepsilon}$ from Lemma \ref{thm.der_carl_est} and the triangle inequality, implies
\begin{align*}
R \int_{\partial\mc{U}\cap \mc{D}_p}\rho\zeta_{a,b;\varepsilon}^p |\mathcal{N} \phi|^2 & \lesssim  \frac{1}{a} \int_{\mc{U}\cap \mc{D}_p} \rho \zeta_{a,b;\varepsilon}^p f_p |\square \phi|^2 + a R^2 \int_{\mc{U}\cap \mc{D}_p} \rho \zeta_{a,b;\varepsilon}^p f_p^{-1}| h \phi|^2 \numberthis \label{eq_g20_MB} \\
& \qquad+ R \int_{\mc{U} \cap \mc{D}_p} \rho \zeta_{a,b;\varepsilon}^p \nabla^\alpha h^\beta \nabla_\alpha \phi \nabla_\beta \phi \\
& \qquad + a R^2 \int_{\mc{U} \cap \mc{D}_p} (\rho + |h^\beta\nabla_\beta \rho|) \zeta_{a,b;\varepsilon}^p f_p^{-1} \nabla_\alpha \phi \nabla^\alpha \phi \\
& \qquad + a R^2 \int_{\mc{U} \cap \mc{D}_p} (\rho + |\nabla^\alpha \rho|) \zeta_{a,b;\varepsilon}^p f_p^{-1} | \nabla_\alpha \phi \cdot h\phi| .
\end{align*}

Using \eqref{eq:3.1_MB}, the LHS can be estimated from below, as follows
\begin{equation*}
R \int_{\Gamma_+} \zeta_{a,b;\varepsilon}^p |\mathcal{N} \phi|^2 \leqslant R \int_{\Gamma_+} \rho \zeta_{a,b;\varepsilon}^p |\mathcal{N} \phi|^2 \leqslant
R \int_{\partial\mc{U} \cap \mc{D}_p} \rho \zeta_{a,b;\varepsilon}^p |\mathcal{N} \phi|^2 \text{,}
\end{equation*}
where we also used the fact that $\Gamma_+ \subset \partial \mc{U} \cap \mc{D}_p$. Then, combining the above and \eqref{eq_g20_MB}, shows that
\begin{equation} \label{eq:g2_MB} 
R \int_{ \Gamma_+ } \zeta_{a,b;\varepsilon}^p |\mathcal{N} \phi|^2 \lesssim \frac{1}{a} \int_{\mc{U}\cap \mc{D}_p} \zeta_{a,b;\varepsilon}^p f_p |\square \phi|^2 + a R^2 \int_{\mathcal{O}_{\sigma/2}(\Gamma_+) \cap (\mc{U} \cap \mc{D}_p)} \zeta_{a,b;\varepsilon}^p f_p^{-1} ( |\nabla_x \phi|^2 + |\partial_t \phi|^2 ) \text{.} 
\end{equation}

We will bound the last term in the RHS of the above estimate, by data on the region $W$.
For this purpose, we first define a cut-off function \( \rho_1\in C^2 ( \Bar{\mc{U}};[0,1] ) \) as follows
\begin{equation} \label{eq:rho_1_MB}
\rho_1(t,x) = \begin{cases}
1 \text{,} \qquad (t,x) \in \mathcal{O}_{\sigma/2}(\Gamma_+) \cap \mc{U} \cap \mc{D}_p \text{,} \\
0 \text{,} \qquad (t,x) \in (\mc{U} \cap \mc{D}_p) \setminus W \text{.}
\end{cases} 
\end{equation}

Then, let \( \eta \) be the function defined as $\eta(t,x):=\rho_1^2 \zeta_{a,b;\varepsilon}^p f_p^{-1}$. Then, we have the following estimates for derivatives of $\eta$
\begin{align*} 
|\nabla_x \eta| = |\nabla_x (\rho_1^2 \zeta_{a,b;\varepsilon}^p f_p^{-1})| & \lesssim \rho_1 \zeta_{a,b;\varepsilon}^p \left( f_p^{-1} |\nabla_x \rho_1| + a R \rho_1 f_p^{-2} \right) \text{,} \numberthis \label{eq_eta_der_MB}\\ 
|\partial_t \eta| & \lesssim a R \rho_1^2  \zeta_{a,b;\varepsilon}^p f_p^{-2} \text{.}
\end{align*}

Now, an application of integration by parts gives us the following
\begin{align*}
\int_{\mc{U}\cap \mc{D}_p}\eta |\nabla_x \phi|^2 & = - \int_{\mc{U}\cap \mc{D}_p}\eta\phi \square \phi + \int_{\mc{U}\cap \mc{D}_p} \phi \partial_t \phi \partial_t \eta + \int_{\mc{U}\cap \mc{D}_p} \eta |\partial_t \phi|^2  - \int_{\mc{U}\cap \mc{D}_p} \phi \nabla_x \phi \cdot \nabla_x \eta \text{.} 
\end{align*} 

Substituting the expression for $\eta$, we get
\begin{align*}
\int_{\mc{U} \cap \mc{D}_p} \rho_1^2 \zeta_{a,b;\varepsilon}^p f_p^{-1} |\nabla_x \phi|^2 & \leqslant  \int_{\mc{U}\cap \mc{D}_p}|\eta \phi\square \phi| + \int_{\mc{U}\cap \mc{D}_p}|\phi \partial_t \phi \cdot \partial_t \eta| + \int_{\mc{U}\cap \mc{D}_p} \eta |\partial_t \phi|^2  \numberthis \label{eq_eta_2_MB}\\
& \qquad - \int_{\mc{U}\cap \mc{D}_p} \phi \nabla_x \phi \cdot \nabla_x \eta \text{.}
\end{align*}

Using \eqref{eq_eta_der_MB}, and also multiplying the coefficient $aR^2$ throughout, the above estimate gives
\begin{align*}
a R^2 \int_{\mc{U}\cap \mc{D}_p} \rho_1^2 \zeta_{a,b;\varepsilon}^p f_p^{-1} | \nabla_x \phi|^2 & \lesssim a R^2 \int_{\mc{U}\cap \mc{D}_p} \rho_1^2\zeta_{a,b;\varepsilon}^p f_p^{-1} | \phi\square \phi|  \numberthis \label{eq_g111_MB} \\
& \qquad + a^2 R^3 \int_{\mc{U}\cap \mc{D}_p}\rho_1^2 \zeta_{a,b;\varepsilon}^p f_p^{-2} |\phi \partial_t \phi | + a R^2 \int_{\mc{U}\cap\mc{D}_p}\rho_1^2\zeta_{a,b;\varepsilon}^p f_p^{-1}|\partial_t \phi|^2 \\
& \qquad  + a R^2 \int_{\mc{U}\cap \mc{D}_p} \rho_1 \zeta_{a,b;\varepsilon}^p \left[ f_p^{-1} |\nabla_x \rho_1| + a R \rho_1 f_p^{-2} \right] | \phi \nabla_x \phi | \text{.}
\end{align*}

Note that, using the Cauchy-Schwarz inequality shows that the following is satisfied
\begin{align*}
a R^2 \int_{\mc{U}\cap \mc{D}_p} & \rho_1 \zeta_{a,b;\varepsilon}^p \left[ f_p^{-1} |\nabla_x \rho_1| + a R \rho_1 f_p^{-2} \right] | \phi \nabla_x \phi | \\
& \lesssim a^2 R^2 \int_{\mc{U}\cap \mc{D}_p} |\nabla_x\rho_1|^2 \zeta_{a,b;\varepsilon}^p f_p^{-1} \phi^2 + R^2 \int_{\mc{U}\cap \mc{D}_p} \rho_1^2 \zeta_{a,b;\varepsilon}^p f_p^{-1} |\nabla_x \phi|^2 \\
& \qquad + a^4 R^4 \int_{\mc{U}\cap \mc{D}_p} \rho_1^2 \zeta_{a,b;\varepsilon}^p f_p^{-3} \phi^2 \text{.}
\end{align*}

We apply Cauchy-Schwarz inequality for the $``\phi \square \phi"$ and the $``\phi \partial_t \phi"$ integral terms in the RHS of \eqref{eq_g111_MB}. Then, \eqref{eq_g111_MB} reduces to
\begin{align*}
\numberthis \label{eq:g123_MB} a R^2 \int_{ \mc{U} \cap \mc{D}_p } & \rho_1^2 \zeta_{a,b;\varepsilon}^p f_p^{-1} |\nabla_x \phi|^2 \\
& \lesssim \frac{1}{a} \int_{\mc{U}\cap \mc{D}_p} \rho_1^2 \zeta_{a,b;\varepsilon}^p f_p |\square \phi|^2  + a R^2 \int_{\mc{U}\cap \mc{D}_p} \rho_1^2 \zeta_{a,b;\varepsilon}^p f_p^{-1} |\partial_t \phi|^2\\
& \quad + a^2 R^2 \int_{\mc{U}\cap \mc{D}_p} (|\nabla_x\rho_1|^2 + |\partial_t \rho_1|^2) \zeta_{a,b;\varepsilon}^p f_p^{-1} \phi^2 + a^4 R^4 \int_{\mc{U}\cap \mc{D}_p} \rho_1^2 \zeta_{a,b;\varepsilon}^p f_p^{-3} \phi^2 \text{.}
\end{align*}

Now, the coefficients of \( \phi^2 \) can be estimated as
\begin{align*}
[ a^2 R^2 (|\nabla_x \rho_1|^2+ |\partial_t \rho_1|^2) f_p^{-1} + a^4 R^4 \rho_1^2 f_p^{-3} ] & \lesssim [ a^2 R^6 (|\nabla_x \rho_1|^2 + |\partial_t \rho_1|^2) f_p^{-3} + a^4 R^4 \rho_1^2 f_p^{-3} ] \\
& \lesssim \ a^4 R^4 \Big( |\nabla_x \rho_1|^2 + |\partial_t \rho_1|^2 + \rho_1^2 \Big) f_p^{-3} \text{.}
\end{align*}

Then, reducing the integral region on the LHS of \eqref{eq:g123_MB} appropriately and using \eqref{eq:rho_1_MB}, we get from \eqref{eq:g123_MB} that
\begin{align*}
a R^2 \int_{\mathcal{O}_{\sigma/2}(\Gamma_+) \cap \mc{U} \cap \mc{D}_p } \rho_1^2 \zeta_{a,b;\varepsilon}^p f_p^{-1} |\nabla_x \phi|^2 & \lesssim \frac{1}{a} \int_{\mc{U}\cap \mc{D}_p} \zeta_{a,b;\varepsilon}^p f_p |\square \phi|^2 + a R^2 \int_W \zeta_{a,b;\varepsilon}^p f^{-1} |\partial_t \phi|^2 \\
& \qquad + a^4 R^4  \int_W \zeta_{a,b;\varepsilon}^p f_p^{-3} \phi^2 \text{.}
\end{align*}
Combining \eqref{eq:g2_MB} and the above estimate, shows that
\begin{align*}
R \int_{ \Gamma_+ } \zeta_{a,b;\varepsilon}^p |\mathcal{N} \phi|^2 \lesssim \frac{1}{a} \int_{\mc{U}\cap \mc{D}_p} \zeta_{a,b;\varepsilon}^p f_p |\square \phi|^2 + a R^2 \int_W \zeta_{a,b;\varepsilon}^p f_p^{-2} |\partial_t \phi|^2 + a^4 R^4  \int_W \zeta_{a,b;\varepsilon}^p f_p^{-3} \phi^2 \text{.}
\end{align*}
This concludes the proof of the theorem.
\end{proof}

\section{Observability} \label{ch4_sec_obs}

In this section, we will use Theorem \ref{thm.carl_int_MB} to prove the main observability result Theorem \ref{thm_obs_main_intro_MB}. We will first prove two preliminary observability estimates
\begin{itemize}
\item Exterior observability: This deals with the case when we apply the Carleman estimate about a point $p \notin \bar{\mc{U}}$.

\item Interior observability: Here, we apply the Carleman estimate around a point $p \in \mc{U}$.

\end{itemize}

Then we combine the two results appropriately, to conclude the proof of Theorem \ref{thm_obs_main_intro_MB}.

We now present an energy estimate result for $\phi$, the solution of the adjoint system\eqref{eq.intro_obs_MB} that will be used to show the above observability estimates.

\begin{proposition}[Energy Estimate I] \label{thm.energy_est_1_MB}
Let $\tau_1 < \tau_2$, and define constants $\mc{M}_0$ and $\mc{M}_1$ as follows
\begin{equation}
\label{eq.energy_est_M_MB} \mc{M}_0 := 1 + \sup_{ \mc{U}_{\tau_1, \tau_2}  } | V | \text{,} \qquad \mc{M}_1 := 1 + \sup_{ \mc{U}_{\tau_1, \tau_2} } | \mc{X}^{ t, x } | \text{.}
\end{equation}
Then, there exist constants $C, C' > 0$, depending on $\mc{U}, \tau_1, \tau_2$, such that
\begin{align}
\label{eq.energy_est_a_MB} \int_{ \mc{U}_{\tau_1} } ( | \nabla_{ t, x } \phi |^2 + \mc{M}_0 \cdot \phi^2 ) &\leq C e^{ C' ( \mc{M}_0^\frac{1}{2} + \mc{M}_1 ) |\tau_1-\tau_2| } \int_{\mc{U}_{\tau_2}} ( | \nabla_{ t, x } \phi |^2 + \mc{M}_0 \cdot \phi^2 ) \text{,} \\
\notag \int_{ \mc{U}_{\tau_2} } ( | \nabla_{ t, x } \phi |^2 + \mc{M}_0 \cdot \phi^2 ) &\leq C e^{ C' ( \mc{M}_0^\frac{1}{2} + \mc{M}_1 ) |\tau_1-\tau_2| } \int_{\mc{U}_{\tau_1}} ( | \nabla_{ t, x } \phi |^2 + \mc{M}_0 \cdot \phi^2 ) \text{.}
\end{align}
for any solution $\phi \in C^2 ( \mc{U} ) \cap C^1 ( \bar{\mc{U}} )$ of \eqref{eq.intro_obs_MB} satisfying $\phi |_{ \partial \mc{U}_{\tau_1, \tau_2} } = 0$.
\end{proposition}
The above proposition is proved using standard energy arguments (see \cite[Proposition 2.21]{Arick}, or \cite{duy_zhang_zua:obs_opt} for related versions).

\subsection{Exterior observability} \label{ssec_ext_obs_MB}
The following result gives us the observability estimate when the point $p \notin \bar{\mc{U}}$.
\begin{theorem} \label{thm_obs_1_MB}
Let $p \in \R^{1+n} \setminus \bar{\mc{U}}$ be fixed. Also, fix $ 0 < \delta \ll 1$. Assume that $\mc{U} \cap \mc{D}_p$ is bounded. Define the following constants
\begin{align*}
M_0 := \sup_{\mc{U} \cap \mc{D}_p} |V|, & \qquad M_1 := \sup_{\mc{U} \cap \mc{D}_p} |\mc{X}^{t,x}|, \numberthis \label{eq_obs_thm_1_MB}\\
R_+ := \sup_{\mc{U} \cap \mc{D}_p} r_p, & \qquad R_- := \inf_{\mc{U}_{t(p)}} r_p \text{.} 
\end{align*}
Next, choose $\varepsilon$ as
\begin{equation}
\varepsilon = \frac{\delta^2}{R_+} \text{.}
\end{equation}
Now, let $\Gamma_+$ and $W$ be defined according to Definition \ref{def_Gamma_MB} and Definition \ref{def_omega_MB}, respectively, using the above $\varepsilon$. That is,
\begin{equation} \label{eq_obs_thm_1s_MB}
\Gamma_+  := \partial \mc{U} \cap \mc{D}_p \cap \left\{ \left( 1 - \frac{\delta^2}{R_+} r_p \right) \mc{N} f_p + \frac{\delta^2}{R_+} f_p \mc{N} r_p > 0 \right\} \text{,} \quad W = \mc{O}_\sigma (\Gamma_+) \cap (\mc{U} \cap \mc{D}_p).
\end{equation}
Then, there exists a constant $C>0$, such that the following is satisfied
\begin{equation}
\int_{\mc{U}_{t(p)}} ( |\nb_{t,x} \phi|^2 + \phi^2 ) \leqslant C \int_{W} ( |\partial_t \phi|^2 + \phi^2) \text{,}
\end{equation}
for any solution $\phi \in C^2(\mc{U}) \cap C^1(\bar{\mc{U}})$ of \eqref{eq.intro_obs_MB} satisfying $\phi|_{\partial \mc{U} \cap \mc{D}_p}=0$.
\end{theorem}
We also need the following energy estimate that is restricted to the exterior region $\mc{D}_p$.
\begin{proposition}[Energy Estimate II] \label{thm_energy_est_2_MB}
Let $p \in \R^{1+n}$ be fixed. Let $M_0, M_1$ be defined as in \eqref{eq_obs_thm_1_MB}. Then for any $\tau \in (\tau_-, \tau_+)$, we have
\begin{equation}
\int_{ \mc{U}_\tau \cap \mc{D}_p } [| \nabla_{ t, x } \phi |^2 + (1+M_0) \phi^2] \leq C e^{ C' (1 + M_0^{\frac{1}{2}} + M_1 ) |\tau - t_p| } \int_{ \mc{U}_{t_p} \cap \mc{D}_p } [ | \nabla_{ t, x } \phi |^2 + (1+M_0) \phi^2 ] ,
\end{equation}
for any solution $\phi \in C^2(\mc{U}) \cap C^1(\bar{\mc{U}})$ of \eqref{eq.intro_obs_MB} satisfying $\phi|_{\partial \mc{U} \cap \mc{D}_p}=0$.
\end{proposition}
Similar to Proposition \ref{thm.energy_est_1_MB}, the above proposition can be proved by using standard energy arguments.

\begin{proof}[Proof of Theorem \ref{thm_obs_1_MB}]
First, note that \eqref{eq.carl_int_domain_MB} is satisfied with $R=R_+$. Now, choose sufficiently large $a$ such that 
\begin{equation} \label{eq_obs_pf_01_MB}
a \gg_{ \mc{U} } R_+ \text{,} \qquad a \gg_{ \mc{U} } \delta^{ - \frac{1}{3} } R_+^\frac{4}{3} M_0^\frac{2}{3} \text{,} \qquad a \gg_{ \mc{U} } \delta^{-2} R_-^{-2} R_+^4 M_1^2 \text{,}
\end{equation}
and $b:= \delta R_+^{-1}$. 

Now, we apply Theorem \ref{thm.carl_int_MB} to the given $a, b, \varepsilon, \mc{U}, p$, and get the following estimate
\begin{align*}
\numberthis \label{eq_obs_pf_MB} \frac{C \delta^2}{R_+^2} \int_{\mc{U} \cap \mc{D}_p} & \zeta^p_{a,b;\varepsilon} (|u_p \partial_{u_p} \phi |^2 + |v_p \partial_{v_p} \phi |^2 + f_p g^{ab}\slashed \nabla^p_a \phi \slashed\nabla^p_b \phi ) + \frac{C \delta a^2}{R_+^2}\int_{\mc{U} \cap \mc{D}_p} \zeta^p_{a,b;\varepsilon} \phi^2 \\
& \leqslant \frac{2}{a} \int_{\mc{U}\cap \mc{D}_p} \zeta^p_{a,b;\varepsilon} f_p |\nb_{\mc{X}} \phi|^2 + \frac{2}{a} \int_{\mc{U}\cap \mc{D}_p} \zeta^p_{a,b;\varepsilon} f_p V^2 |\phi|^2 \\
& \qquad + a R_+^2 \int_W \zeta_{a,b;\varepsilon} f_p^{-1} |\partial_t \phi|^2 + a^4 R_+^4 \int_W \zeta^p_{a,b;\varepsilon} f_p^{-3} \phi^2 . 
\end{align*}
We use $I_{\mc{X}}$, $I_V$, and $I_W$ to denote the terms on the RHS of the above estimate:
\begin{align*}
I_{\mc{X}} & := \frac{2}{a} \int_{\mc{U}\cap \mc{D}_p} \zeta^p_{a,b;\varepsilon} f_p |\nb_{\mc{X}} \phi|^2, \\
I_V & :=  \frac{2}{a} \int_{\mc{U}\cap \mc{D}_p} \zeta^p_{a,b;\varepsilon} f_p V^2 |\phi|^2, \\
I_W & := a R_+^2 \int_W \zeta^p_{a,b;\varepsilon} f_p^{-1} |\partial_t \phi|^2 + a^4 R_+^4 \int_W \zeta^p_{a,b;\varepsilon} f_p^{-3} \phi^2.
\end{align*}
Then, using the fact that $f_p \leq r_p^2 \leq R_+^2$, and \eqref{eq_obs_pf_01_MB}, we get
\begin{align*}
I_V = \frac{2}{a} \int_{\mc{U}\cap \mc{D}_p} \zeta^p_{a,b;\varepsilon} f_p |V \phi|^2 \leqslant \frac{2 R_+^2 M_0^2}{a} \int_{\mc{U}\cap \mc{D}_p} \zeta^p_{a,b;\varepsilon} \phi^2 \ll \frac{\delta a^2}{R_+^2} \int_{\mc{U}\cap \mc{D}_p} \zeta^p_{a,b;\varepsilon} \phi^2 ,
\end{align*}
which means $I_V$ can be absorbed into the LHS of \eqref{eq_obs_pf_MB}. This implies that
\begin{equation}
C \int_{\mc{U} \cap \mc{D}_p} \zeta^p_{a,b;\varepsilon} \left( \frac{ \delta^2}{R_+^2}(|u_p \partial_{u_p} \phi |^2 + |v_p \partial_{v_p} \phi |^2 + f_p g^{ab}\slashed \nabla^p_a \phi \slashed\nabla^p_b \phi ) + \frac{\delta a^2}{R_+^2} \phi^2 \right)\leqslant I_{\mc{X}} + I_W . \label{eq_obs_pf_2_MB}
\end{equation}
Note that, the weight present in $I_\mc{X}$ is different from the weight present in the first order term on the LHS. This prevents us from absorbing $I_\mc{X}$ into the LHS appropriately. To solve this issue, we decompose the domain $\mc{U} \cap \mc{D}_p$ as follows
\begin{align*} \numberthis \label{eq_obs_pf_3_MB}
\mc{U}_\leq &:= \mc{U} \cap \mc{D}_p \cap \left\{ \frac{ f_p }{ ( 1 + \varepsilon u_p ) ( 1 - \varepsilon v_p ) } \leq \frac{ R_-^2 }{ 64 } \right\} \text{,} \\
\notag \mc{U}_> &:= \mc{U} \cap \mc{D}_p \cap \left\{ \frac{ f_p }{ ( 1 + \varepsilon u_p ) ( 1 - \varepsilon v_p ) } > \frac{ R_-^2 }{ 64 } \right\} \text{.}
\end{align*}
Since on $\mc{U}_>$, we have
\begin{align*}
v_p = \frac{ f_p }{ - u_p } \gtrsim \frac{ R_-^2 }{ R_+ } \text{,} \qquad - u_p = \frac{ f_p }{ v_p } \gtrsim \frac{ R_-^2 }{ R_+ } \text{,}
\end{align*}
then \eqref{eq_obs_pf_2_MB} implies that
\begin{equation}
C \int_{\mc{U}_>} \zeta^p_{a,b;\varepsilon} \left[ \frac{ \delta^2 R_-^2 }{R_+^2}(-u_p |\partial_{u_p} \phi |^2 + v_p |\partial_{v_p} \phi |^2 + v_p g^{ab}\slashed \nabla^p_a \phi \slashed\nabla^p_b \phi ) + \frac{\delta a^2}{R_+^2} \phi^2 \right] \leqslant I_{\mc{X}} + I_W ,\label{eq_obs_pf_4_MB}
\end{equation}
where we note that we have shrunk the integral region in the LHS. 
Now, let us use the following notations:
\[ I_{\mc{X},>} := \frac{2}{a} \int_{\mc{U}_>} \zeta^p_{a,b;\varepsilon} f_p |\nb_{\mc{X}} \phi|^2, \qquad I_{\mc{X},\leqslant} := \frac{2}{a} \int_{\mc{U}_\leqslant} \zeta^p_{a,b;\varepsilon} f_p |\nb_{\mc{X}} \phi|^2. \]
Then, note that
\begin{align*}
I_{\mc{X},>} &\leq \frac{ R_+ M_1^2 }{ a } \int_{ \mc{U}_> } \zeta^p_{ a, b; \varepsilon } ( - u_p | \partial_{ u_p } \phi |^2 + v_p | \partial_{ v_p } \phi |^2 + v_p g^{ab} \nasla^p_a \phi \nasla^p_b \phi ) \\
&\ll_{ \mc{U} } \frac{ \delta^2 R_-^2 }{ R_+^3 } \int_{ \mc{U}_> } \zeta^p_{ a, b; \varepsilon } ( - u_p | \partial_{ u_p } \phi |^2 + v_p | \partial_{ v_p } \phi |^2 + v_p g^{ab} \nasla^p_a \phi \nasla^p_b \phi ) \text{,}
\end{align*}
where we also used \eqref{eq_obs_pf_01_MB}.
Thus, after absorbing this term, \eqref{eq_obs_pf_4_MB} reduces to
\begin{equation} \label{eq_obs_pf_42_99_MB}
C \int_{\mc{U}_>} \zeta^p_{a,b;\varepsilon} \left[ \frac{ \delta^2 R_-^2 }{R_+^2}(-u_p |\partial_{u_p} \phi |^2 + v_p |\partial_{v_p} \phi |^2 + v_p g^{ab}\slashed \nabla^p_a \phi \slashed\nabla^p_b \phi ) + \frac{\delta a^2}{R_+^2} \phi^2 \right] \leqslant I_{\mc{X},\leqslant} + I_W.
\end{equation}
Since $\partial \mc{U}$ is timelike, for any $\tau \in \R$ satisfying $| \tau - t ( p ) | \leq \frac{1}{4} R_-$, we have
\begin{equation}
r_p |_{ \mc{U}_\tau } \geq \frac{ 3 R_- }{4} \text{,} \qquad -u_p |_{ \mc{U}_\tau } \geq \frac{ R_- }{4} \text{,} \qquad v_p |_{ \mc{U}_\tau } \geq \frac{ R_- }{4} \text{.}
\end{equation}
This implies that, for $\tau \in \R$ satisfying $| \tau - t ( p ) | \leq \frac{1}{4} R_-$, we have
\[ 
\left. \frac{ f_p }{ ( 1 + \varepsilon u_p ) ( 1 - \varepsilon v_p ) } \right|_{ \mc{U}_\tau } \geq \frac{ R_-^2 }{ 16 } \text{,} \qquad \zeta_{ a, b; \varepsilon }^p |_{ \mc{U}_\tau } \geq \left( \frac{ R_- }{4} e^\frac{ b R_- }{4} \right)^{ 4 a } \text{.}
\]
Due to \eqref{eq_obs_pf_3_MB}, this just means that 
\[ \mc{U} \cap \left\{  - \frac{ R_- }{4} < t_p < \frac{ R_- }{4} \right\} \subseteq \mc{U}_> \text{.} \]
Then, applying Fubini's theorem, using \eqref{eq_obs_thm_1_MB}, and\eqref{eq_obs_pf_01_MB}, we get that \eqref{eq_obs_pf_42_99_MB} reduces to
\begin{equation} \label{eq_obs_pf_40_MB}
\frac{ C \delta^2 R_-^3 }{ R_+^3 } \left( \frac{ R_- }{ 4 } e^\frac{ b R_- }{4} \right)^{4 a} \int_{ t (p) - \frac{ R_- }{4} }^{ t (p) + \frac{ R_- }{4} } \int_{ \mc{U}_\tau } (| \nabla_{ t, x } \phi |^2 + ( 1 + M_0 ) \phi^2 ) d \tau \leqslant I_{\mc{X},\leqslant} + I_W \text{.}
\end{equation}
Now, using Proposition \ref{thm.energy_est_1_MB} shows that, for $| \tau - t ( p ) | \leq \frac{1}{4} R_-$, we have
\begin{equation} \label{eq_obs_pf_41_MB}
\int_{ \mc{U}_{ t (p) } } [ | \nabla_{ t, x } \phi |^2 + ( 1 + M_0 ) \phi^2 ] \leq C e^{ C_1 ( 1 + M_0^\frac{1}{2} + M_1 ) R_- } \int_{ \mc{U}_\tau } [ | \nabla_{ t, x } \phi |^2 + ( 1 + M_0 ) \phi^2 ] \text{.}
\end{equation}
Note that on $\mc{U}_\leq$, we have
\begin{align*}
\zeta^p_{ a, b; \varepsilon } & \leq \left( \frac{ R_- }{ 8 } e^\frac{ b R_- }{ 8 } \right)^{ 4 a }  \text{,} \qquad f_p \leq R_-^2,\\
I_{\mc{X}, \leq} & \leq \frac{ C' \delta^2 R_-^4 }{ R_+^4 } \left( \frac{ R_- }{ 8 } e^\frac{ b R_- }{8} \right)^{4 a} \int_{ \mc{U} \cap \mc{D}_p } | \nabla_{ t, x } \phi |^2 ,
\end{align*}
where we also used \eqref{eq_obs_pf_01_MB}. Because $\partial \mc{U}$ is timelike, we have that on $\mc{U} \cap \mc{D}_p$ 
\[ |t_p| \leqslant R_+ ,\]
which after using Proposition \ref{thm_energy_est_2_MB} and Fubini's theorem, gives us
\begin{equation} \label{eq_obs_pf_42_MB}
I_{\mc{X}, \leq} \leq \frac{ C' \delta^2 R_-^4 }{ R_+^3 } \left( \frac{ R_- }{ 8 } e^\frac{ b R_- }{8} \right)^{4 a} e^{ C_1 ( 1 + M_0^\frac{1}{2} + M_1 ) R_+ } \int_{ \mc{U}_{ t (p) } } [ | \nabla_{ t, x } \phi |^2 + ( 1 + M_0 ) \phi^2 ]. 
\end{equation}
Now, combining \eqref{eq_obs_pf_40_MB}-\eqref{eq_obs_pf_42_MB}, and then using \eqref{eq_obs_pf_01_MB}, we get
\begin{equation} \label{eq_obs_pf_5_MB}
\frac{ C \delta^2 R_-^4 }{ R_+^3 } \left( \frac{ R_- }{ 4 } \right)^{4 a} e^{ - C_1 ( 1 + M_0^\frac{1}{2} + M_1 ) R_- } \int_{ \mc{U}_{ t (p) } } ( | \nabla_{ t, x } \phi |^2 + \phi^2 ) \leqslant I_W \text{.}
\end{equation}
For $I_W$, note that
\begin{align*} 
I_W & \leqslant a^4 R_+^4 \int_W \zeta_{a,b;\varepsilon} f_p^{-3} (|\partial_t \phi|^2 + \phi^2), \numberthis \label{eq_obs_pf_511_MB}\\
a^4 R_+^4 f_p^{-3} \zeta^p_{ a, b; \varepsilon } & \leqslant a^4 2^{4a} R_+^{4a-2}.
\end{align*}
Then, \eqref{eq_obs_pf_5_MB} implies that
\begin{equation} \label{eq_obs_pf_5111_MB}
\int_{ \mc{U}_{ t (p) } } ( | \nabla_{ t, x } \phi |^2 + \phi^2 ) \leqslant \frac{C a^4 2^{16a} R_+^{-3}}{\delta^2} \left( \frac{ R_+ }{ R_- } \right)^{4a + 4} \int_W ( |\partial_t \phi|^2 + \phi^2 ) \text{.}
\end{equation}
This completes the proof of the theorem.
\end{proof}

\begin{remark}
The above proof allows us to extract an explicit form of the observability constant. Indeed, combining \eqref{eq_obs_pf_01_MB} and \eqref{eq_obs_pf_5111_MB} gives us the exact expression of this constant. We do not present it in Theorem \ref{thm_obs_1_MB} (or in Theorem \ref{thm_obs_main_intro_MB}) to keep the statement concise.
\end{remark}

\subsection{Interior observability} \label{ssec_int_obs_MB}
We can also prove a similar observability estimate when the point $p \in \mc{U}$. Due to a technical issue, now we need to apply the Carleman estimate around two points. Essentially, the Carleman weight vanishes at the observation point which now lies inside the domain $\mc{U}$. Thus, we cannot control the $H^1$-norm of $\phi$ using the Carleman estimate. To solve this, we apply the estimate around two points and then add them together to obtain the contribution of the whole domain.

\begin{theorem} \label{thm_obs_2_MB}
Fix $ 0 < \delta \ll 1$. Let $p_1, p_2 \in \mc{U}$ be such that
\begin{equation}
p_1 \neq p_2, \qquad t(p_1) = t(p_2) = t_0.
\end{equation}
Assume that $\mc{U} \cap (\mc{D}_{p_1} \cup \mc{D}_{p_2})$ is bounded. Define the following constants
\begin{align*}
M_0 := \max_{i=1,2} \sup_{\mc{U} \cap \mc{D}_{p_i}} |V|, & \qquad M_1 := \max_{i=1,2} \sup_{\mc{U} \cap \mc{D}_{p_i}} |\mc{X}^{t,x}|, \numberthis \label{eq_obs_thm_2_MB}\\
R_+ := \max_{i=1,2} \sup_{\mc{U} \cap \mc{D}_{p_i}} r_p, & \qquad R_- := \frac{1}{2} |x(p_2) - x(p_1)| \text{.} 
\end{align*}
Now, define
\begin{align}
\Gamma_+^i & := \partial \mc{U} \cap \mc{D}_{p_i} \cap \left\{ \left( 1 - \frac{\delta^2}{R_+} r_{p_i} \right) \mc{N} f_{p_i} + \frac{\delta^2}{R_+} f_{p_i} \mc{N} r_{p_i}  >0 \right\}, \\
W_i & = \mc{O}_\sigma (\Gamma_+^i) \cap (\mc{U} \cap \mc{D}_{p_i}).
\end{align}
Then, there exists a constant $C>0$, such that the following is satisfied
\begin{equation}
\int_{\mc{U}_{t(p)}} ( |\nb_{t,x} \phi|^2 + \phi^2 ) \leqslant C \sum_{i=1,2}\int_{W_i} ( |\partial_t \phi|^2 + \phi^2) \text{,}
\end{equation}
for any solution $\phi \in C^2(\mc{U}) \cap C^1(\bar{\mc{U}})$ of \eqref{eq.intro_obs_MB} satisfying $\phi|_{\partial \mc{U} \cap (\mc{D}_{p_1} \cup \mc{D}_{p_2})}$.
\end{theorem}
The proof of this above theorem is analogous to the proof of Theorem \ref{thm_obs_1_MB}. We mainly outline a sketch, only providing details where new ideas are used.

\begin{proof}
We note that \eqref{eq.carl_int_domain_MB} is satisfied for $p_1,p_2$ with $R=R_+$. For convenience, we use the notation $\mc{U}^i := \mc{U} \cap \mc{D}_{p_i}$. We choose large enough $a$ such that \eqref{eq_obs_pf_01_MB} is satisfied, and also choose $\varepsilon := \delta^2 R_+^{-1} $ and $b:= \delta R_+^{-1}$. Following a similar argument as the one used to show \eqref{eq_obs_pf_2_MB}, applying Theorem \ref{thm.carl_int_MB} to $p_i$, for $i \in \{ 1,2\}$, shows that 
\begin{equation}
C \int_{\mc{U}^i} \zeta^{p_i}_{a,b;\varepsilon} \left( \frac{ \delta^2}{R_+^2}(|u_{p_i} \partial_{u_{p_i}} \phi |^2 + |v_{p_i} \partial_{v_{p_i}} \phi |^2 + f_{p_i} g^{ab}\slashed \nabla^{p_i}_a \phi \slashed\nabla^{p_i}_b \phi ) + \frac{\delta a^2}{R_+^2} \phi^2 \right)\leqslant I_{\mc{X}}^i + I_W^i, \label{eq_obs_int_pf_2_MB}
\end{equation}
where
\[I_{\mc{X}}^i := \frac{2}{a} \int_{\mc{U}^i} \zeta^{p_i}_{a,b;\varepsilon} f_{p_i} |\nb_{\mc{X}} \phi|^2, \qquad I_W^i :=  a R_+^2 \int_W \zeta^{p_i}_{a,b;\varepsilon} f_{p_i}^{-1} |\partial_t \phi|^2 + a^4 R_+^4 \int_W \zeta^{p_i}_{a,b;\varepsilon} f_{p_i}^{-3} \phi^2. \]
Splitting $\mc{U}^i$ as before, into 
\begin{align*} \numberthis \label{eq_obs_int_pf_3_MB}
\mc{U}_\leq^i &:= \mc{U}^i \cap \left\{ \frac{ f_{p_i} }{ ( 1 + \varepsilon u_{p_i} ) ( 1 - \varepsilon v_{p_i} ) } \leq \frac{ R_-^2 }{ 64 } \right\} \text{,} \\
\mc{U}_>^i &:= \mc{U}^i \cap \left\{ \frac{ f_{p_i} }{ ( 1 + \varepsilon u_{p_i} ) ( 1 - \varepsilon v_{p_i} ) } > \frac{ R_-^2 }{ 64 } \right\} \text{,}
\end{align*}
we get that \eqref{eq_obs_int_pf_2_MB} reduces to
\begin{equation} \label{eq_obs_int_pf_33_MB}
C \int_{\mc{U}^i_>} \zeta^{p_i}_{a,b;\varepsilon} \left[ \frac{ \delta^2 R_-^2 }{R_+^2}(-u_{p_i} |\partial_{u_{p_i}} \phi |^2 + v_{p_i} |\partial_{v_{p_i}} \phi |^2 + v_{p_i} g^{ab}\slashed \nabla^{p_i}_a \phi \slashed\nabla^{p_i}_b \phi ) + \frac{\delta a^2}{R_+^2} \phi^2 \right] \leqslant I_{\mc{X},\leqslant}^i + I_W^i.
\end{equation}

Now, note that because $\partial \mc{U}$ is timelike, for any $|\tau - t(p)| < \frac{1}{4} R_-$, and 
\[ \mc{V}^i_\tau := \mc{U}^i_\tau \cap \left\{ r_{p_i} > \frac{3 R_-}{4} \right\}, \]
we get that $\mc{V}^i_\tau \subset \mc{U}^i_>$. Thus, we get from \eqref{eq_obs_int_pf_33_MB} that
\begin{equation} \label{eq_obs_int_pf_40_MB}
\frac{ C \delta^2 R_-^3 }{ R_+^3 } \left( \frac{ R_- }{ 4 } e^\frac{ b R_- }{4} \right)^{4 a} \int_{ t (p) - \frac{ R_- }{4} }^{ t (p) + \frac{ R_- }{4} } \int_{ \mc{V}^i_\tau } (| \nabla_{ t, x } \phi |^2 + ( 1 + M_0 ) \phi^2 ) d \tau \leqslant I_{\mc{X},\leqslant}^i + I_W^i \text{.}
\end{equation}

Since for $|\tau - t(p)| < \frac{1}{4} R_-$, we have $\mc{V}_\tau^1 \cup \mc{V}_\tau^2 = \mc{U}_\tau $, the above estimate implies that
\begin{align*}
\frac{ C \delta^2 R_-^3 }{ R_+^3 } \left( \frac{ R_- }{ 4 } e^\frac{ b R_- }{4} \right)^{4 a} \int_{ t (p) - \frac{ R_- }{4} }^{ t (p) + \frac{ R_- }{4} } \int_{ \mc{U}_\tau } (| \nabla_{ t, x } \phi |^2 + ( 1 + M_0 ) \phi^2 ) d \tau \leqslant \sum_{i=1,2} [ I_{\mc{X},\leqslant}^i + I_W^i ] \text{.}
\end{align*}

Following a procedure analogous to \eqref{eq_obs_pf_5_MB}, we obtain
\begin{equation} \label{eq_obs_int_pf_5_MB}
\frac{ C \delta^2 R_-^4 }{ R_+^3 } \left( \frac{ R_- }{ 4 } \right)^{4 a} e^{ - C_1 ( 1 + M_0^\frac{1}{2} + M_1 ) R_- } \int_{ \mc{U}_{ t (p) } } ( | \nabla_{ t, x } \phi |^2 + \phi^2 ) \leqslant \sum_{i=1,2} I_W^i \text{.}
\end{equation}
Estimating the weights present in the integrand of $I_W^i$ similar to \eqref{eq_obs_pf_511_MB}, completes the proof of the theorem.
\end{proof}

\subsection{Proof of Theorem \ref{thm_obs_main_intro_MB}}
Now we are ready to prove the main observability result Theorem \ref{thm_obs_main_intro_MB}.
\begin{proof}
Throughout the proof, we let $p = (t_0,x_0)$. Note that, $\tau_+ - \tau_- > R_+ + R_-$ and the choice of $t_0$ implies that
\[ \partial\mc{U} \cap \mc{D}_p \subset \partial\mc{U}_{\tau_-,\tau_+}, \qquad \mc{U} \cap \mc{D}_p \subset \mc{U}_{\tau_-,\tau_+}. \]
We will divide the proof into parts, depending on the location of $p$ with respect to $\mc{U}$.

Firstly, let $p \notin \bar{\mc{U}}$. Then, Theorem \ref{thm_obs_1_MB} implies that
\begin{equation} \label{eq_obs_main_A_MB}
\int_{\mc{U}_{t(p)}} ( |\nb_{t,x} \phi|^2 + \phi^2 ) \leqslant C \int_{W} ( |\partial_t \phi|^2 + \phi^2) \text{,}
\end{equation}
where $W$ is given by \eqref{eq_obs_thm_1s_MB}.
Now, when $\delta \searrow 0$, we have 
\begin{align*}
\left\{ \left( 1 - \frac{\delta^2}{R_+} r_p \right) \mc{N} f_p + \frac{\delta^2}{R_+} f_p \mc{N} r_p > 0 \right\} \rightarrow \{ \mc{N} f_p > 0 \},
\end{align*}
That is, for sufficiently small $\delta$, we have $W \subset W'$. Then, using Proposition \ref{thm.energy_est_1_MB} to change the integral region $\mc{U}_{t(p)} \rightarrow \mc{U}_{\tau_\pm}$ on the LHS of \eqref{eq_obs_main_A_MB}, and taking $\delta$ sufficiently small, shows that \eqref{eq_obs_main_0_MB} holds for this case.

For $p \in \mc{U}$, we choose two distinct points $p_1,p_2 \in \mc{U}$, such that $t(p_1)=t(p_2)=t(p)$ and apply Theorem \ref{thm_obs_2_MB}. Now, if $\delta$ is small enough and $p_1,p_2$ are close enough then $ W_1 \cup W_2 \subset W'$. Again, using Proposition \ref{thm.energy_est_1_MB} completes the proof of this case.

For $p \in \partial \mc{U}$, we consider a point $\tilde{p}\notin\bar{\mc{U}}$ close to $p$. Then applying Theorem \ref{thm_obs_1_MB} to $\tilde{p}$ and using a similar argument as before completes the proof of this case.
\end{proof}

\providecommand{\bysame}{\leavevmode\hbox to3em{\hrulefill}\thinspace}
\providecommand{\MR}{\relax\ifhmode\unskip\space\fi MR }
\providecommand{\MRhref}[2]{%
  \href{http://www.ams.org/mathscinet-getitem?mr=#1}{#2}
}
\providecommand{\href}[2]{#2}

\end{document}